\documentclass{article}

\usepackage{amssymb,amsmath,color}
\usepackage{amsthm}
\usepackage{url}

\usepackage{graphicx}

\usepackage{amsmath}
\usepackage{tikz}
\usepackage{algorithm}
\usepackage{algpseudocode}

\definecolor{red}{rgb}{1,0,0}
\definecolor{blue}{rgb}{.2,.2,.8}

\newtheorem{theorem}{Theorem}[section]
\newtheorem{corollary}[theorem]{Corollary}

\theoremstyle{definition}

\begin{document}

\title{Fast algorithm for generating\\ ascending compositions}
\author{Mircea Merca\\
	\footnotesize Constantin Istrati Technical College\\
	\footnotesize Grivitei 91, 105600 Campina, Romania\\
	\footnotesize mircea.merca@profinfo.edu.ro
}
\date{}
\maketitle
%\ccom{} \mcom{}

\begin{abstract} 
In this paper we give a fast algorithm to generate all partitions of a positive integer $n$. Integer partitions may be encoded as either ascending or descending compositions for the purposes of systematic generation. It is known that the ascending composition generation algorithm is substantially more efficient than its descending composition counterpart. Using tree structures for storing the partitions of integers, we develop a new ascending composition generation algorithm which is substantially more efficient than the algorithms from the literature.
\end{abstract}

\section{Introduction}
\label{intro}
Any positive integer $n$ can be written as a sum of one or more positive integers $\lambda_i$, $n=\lambda_1+\lambda_2+\cdots+\lambda_k$. 
If the order of integers $\lambda_i$ does not matter, this representation is known as an integer partition; otherwise, it is know as a composition. 
When $\lambda_1\le \lambda_2\le\cdots\le\lambda_k$, we have an ascending composition. 
If $\lambda_1\ge \lambda_2\ge\cdots\ge\lambda_k$ then we have a descending composition.
We notice that more often than not there appears the tendency of defining partitions as descending composition and this is also the convention used in this paper. In order to indicate that $\lambda=[\lambda_1,\lambda_2,\ldots,\lambda_k]$ is a partition of $n$, we use the notation  $\lambda\vdash n$ introduced by G. E. Andrews \cite{And76}. 
The number of all partitions of a positive integer $n$ is denoted by $p(n)$ (sequence $A000041$ in OEIS \cite{Slo11}). 
%For instance, the seven partitions of integer $5$ are $[5]$, $[4,1]$, $[3,2]$, $[3,1,1]$, $[2,2,1]$, $[2,1,1,1]$ and $[1,1,1,1,1]$. 
%This way of representing the partitions is known as the standard representation.

The choice of the way in which the partitions are represented is crucial for the efficiency of their generating algorithm.
%Although the partitions are fundamentally unordered, they come to be defined in more and more concrete terms as descending compositions.
Kelleher \cite{Kel06,Kel09} approaches the problem of generating partitions through ascending compositions and proves that the algorithm \scshape AccelAsc \normalfont is more efficient than any other previously known algorithms for generating integer partitions.

In this study we show that the tree structures presented by Lin \cite{Lin05} can be used to efficiently generate ascending compositions in standard representation. 
The idea of using tree structures to store all partitions of an integer is based on the fact that two partitions of the same integers could have more common parts. 
Lin \cite{Lin05} created the tree structures according to the following rule: 
the root of the tree is labeled with $(1,n)$, and $(x',y')$ is a child of the node $(x,y)$ if and only if
$$x'\in\left\{x,x+1,\ldots,\left\lfloor \frac{y}{2}\right\rfloor,y\right\}\qquad \mbox{and}\qquad y'=y-x'\ .$$
If $x'=y$ then $(x',y')=(y,0)$ is a leaf node. It is obvious that any leaf node has the form $(x,0)$, $0<x\leq n$. 

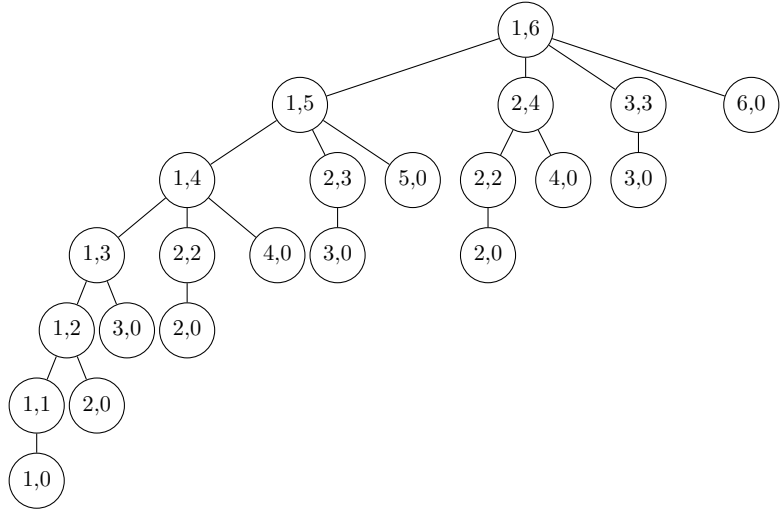
\begin{figure}[H]
	\caption{A partition tree of integer $6$}
	\label{Fig1}
	\centering
	\begin{tikzpicture}
	[level distance=10mm,
	every node/.style={draw,circle,scale=0.85},
	level 1/.style={sibling distance=15mm},
	level 2/.style={sibling distance=10mm},
	level 3/.style={sibling distance=12mm},
	level 4/.style={sibling distance=8mm}]
	\node {1,6}
	child 
	{
		node {1,5}
		child 
		{
			node {1,4}
			child 
			{
				node {1,3}
				child 
				{
					node {1,2}
					child 
					{
						node{1,1}
						child {node{1,0}}
					}
					child {node{2,0}}
				}
				child {	node {3,0} }
			}
			child 
			{
				node {2,2}
				child {node{2,0}}
			}
			child { node {4,0} }
		}
		child[missing]
		child 
		{
			node {2,3}
			child {node {3,0}}
		}
		child { node {5,0} }
	}
	child[missing]
	child 
	{
		node {2,4}
		child 
		{
			node {2,2}
			child {node {2,0}}
		}
		child {node {4,0}}
	}
	child
	{
		node {3,3}
		child {node {3,0}}
	}
	child { node {6,0} }
	;
	\end{tikzpicture}   
\end{figure}

Lin \cite{Lin05} proved that $[x_1,x_2,\ldots,x_k]$ with $x_1\le x_2\le\cdots\le x_k$ is an ascending composition of $n$ if and only if $(x_0,y_0),(x_1,y_1),\ldots,(x_k,y_k)$ is a path from the root $(x_0,y_0)$ to a leaf $(x_k,y_k)$ and then, basing on this, it is shown that the total number of nodes needed to store all the partitions of an integer is twice the number of leaf nodes in its partition tree. Obviously, the number of leaf nodes in partition tree of $n$ is $p(n)$, namely the number of the partitions of $n$.

It is well-known that any ordered tree can be converted into a binary tree by changing the links between the nodes.  
Converting the tree for storing integer partitions, we get a binary tree. Deleting then the root of this binary tree we get a strict binary tree that stores all the partitions of an integer. Applying these conversions to the tree from figure \ref{Fig1}, we obtain the strict binary tree in figure \ref{Fig2}.

In order to generate the ascending compositions we use a depth-first traversal of partition strict binary tree. 
When we reach a leaf node we list from the path that connects the root node with the leaf node only the leaf node and the nodes that are followed by the left child. 
For instance, $(1,5)(1,4)(1,3)(2,2)(2,0)$ is a path that connects the root node $(1,5)$ to the leaf node $(2,0)$. From this path node $(1,3)$ is deleted when listing because it is followed by the node $(2,2)$ which is its right child. Keeping from every remained pair only the first value, we get the ascending composition $[1,1,2,2]$. 

\begin{figure}[H]
	\caption{A partition strict binary tree of integer $6$}
	\label{Fig2}
	\centering
	\begin{tikzpicture}
	[level distance=10mm,
	every node/.style={draw,circle,scale=0.85},
	level 1/.style={sibling distance=50mm},
	level 2/.style={sibling distance=24mm},
	level 3/.style={sibling distance=16mm},
	level 4/.style={sibling distance=8mm},
	level 5/.style={sibling distance=8mm}]
	\node {1,5}
	child 
	{
		node {1,4}
		child 
		{
			node {1,3}
			child 
			{
				node {1,2}
				child 
				{
					node {1,1}
					child 
					{
						node {1,0}		    		
					}
					child 
					{
						node {2,0}
					}
				}
				child 
				{
					node {3,0}
				}
			}
			child 
			{
				node {2,2}
				child 
				{
					node {2,0}
				}
				child 
				{
					node {4,0}
				}
			}
		}
		child 
		{
			node {2,3}
			child 
			{
				node {3,0}
			}
			child 
			{
				node {5,0}
			}
		}
	}
	child
	{
		node {2,4}
		child 
		{
			node {2,2}
			child 
			{
				node {2,0}
			}
			child 
			{
				node {4,0}
			}
		}
		child 
		{
			node {3,3}
			child 
			{
				node {3,0}
			}
			child 
			{
				node {6,0}
			}
		}    
	}
	;
	\end{tikzpicture}   
\end{figure}
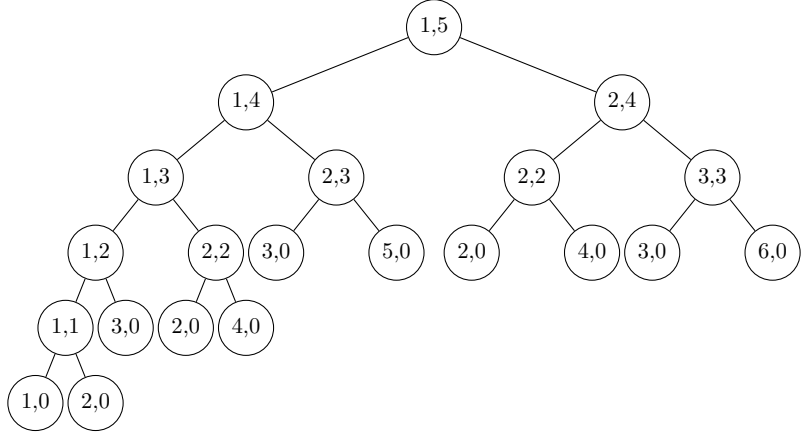

Considering, on the one hand, the way in which the partition tree of $n$ was created, and on the other hand, the rule according to which we can convert any ordered tree into a binary tree, we can deduce the rule according to which we can directly create the partition strict binary tree of $n$. 
The root of partition strict binary tree of $n$ is labeled with $(1,n-1)$, the node $(x_l,y_l)$ is the left child of the node $(x,y)$ if and only if
\begin{equation} \nonumber
x_l = 
\begin{cases} 
x, & \mbox{if\ } 2x\leq y, \\
y, & \mbox{otherwise}
\end{cases}
\mbox{\qquad and \qquad } y_l=y-x_l \ , 
\end{equation}
and the node $(x_r,y_r)$ is the right child of the node $(x,y)$ if and only if
\begin{equation} \nonumber 
x_r = 
\begin{cases} 
x+1, & \mbox{if\ } 2+x\leq y, \\
x+y, & \mbox{otherwise}
\end{cases}
\mbox{\qquad and \qquad } y_r=x+y-x_r \ . 
\end{equation}

%%%%%%%%%%%%%%%%%%%%%%%%%%%%%%%%%%%%%%%%%%%%%%%%%%%%%%%%%%%%

\section{A special case of restricted integer partitions}

The integer partitions in which the parts are at least as large as $m$ represent an example of partition with restrictions. 
We denote by $p(n,m)$ the number of partitions $\lambda \vdash n$ in which the parts are at least as large as $m$.

A special case of partitions that have parts at least as large as $m$ is the partition $\lambda \vdash n$ with the property  $\lambda_1 \ge t\cdot\lambda_2$, where $t$ is a positive integer. 
We consider that the partition $[n]$ has this property and we denote the number of these partitions by $p^{(t)}(n,m)$. 
For instance, $p^{(2)}(15,3)=7$, namely $15$ has seven partitions with parts as least as large as $3$ in which the first part is at least twice the second part: $[15]$, $[12,3]$, $[11,4]$, $[10,5]$, $[9,3,3]$, $[8,4,3]$ and $[6,3,3,3]$.

When $\left\lfloor \frac{n}{t+1}\right\rfloor<m\le n$, it is easily seen that $p^{(t)}(n,m)=1$. 

\begin{theorem}\label{th6A}
	Let $n$, $m$ and $t$ be positive integers so that $m\le \left\lfloor \frac{n}{t+1}\right\rfloor$. Then
	\begin{equation}\label{EQ1135A}
	p^{(t)}(n,m)=p^{(t)}(n-m,m)+p^{(t)}(n,m+1)\ .
	\end{equation}
\end{theorem}

\begin{proof} If $\lambda \vdash n$ is a partition in which the parts are at least as large as $m$ and $\lambda_1 \ge t\cdot\lambda_2$, then the smallest part from $\lambda$ is $m$ or is at least as large as $m+1$. 
	The number of partitions $\lambda \vdash n$ with the smallest part $m$ and $\lambda_1 \ge t\cdot\lambda_2$ is $p^{(t)}(n-m,m)$. 
	Considering that $p^{(t)}(n,m+1)$ is the number of partitions $\lambda \vdash n$ where the smallest part is at least as large as $m+1$ and $\lambda_1 \ge t\cdot\lambda_2$, the theorem is proved.
\end{proof}

For instance, $$p^{(2)}(15,3)=p^{(2)}(12,3)+p^{(2)}(15,4)=4+3=7\ .$$ The number $12$ has four partitions with parts at least as large as $3$ in which the first part is twice the second part: $[12]$, $[9,3]$, $[8,4]$ and $[6,3,3]$. The number $15$ has three partitions with parts at least as large as $4$ in which the first part is twice the second part: $[15]$, $[11,4]$ and $[10,5]$.

\begin{theorem}\label{th6C}
	Let $n$, $m$ and $t$ be positive integers so that $m \le \left\lfloor \frac{n}{t+1}\right\rfloor$. Then
	$$
	p^{(t)}(n,m) =1+\sum_{k=m}^{\left\lfloor \frac{n}{t+1}\right\rfloor} p^{(t)}(n-k,k)\ .
	$$
\end{theorem}

\begin{proof} We expand the term $p^{(t)}(n,m+1)$ from the relation \eqref{EQ1135A} and take into account that $p^{(t)}\left(n,\left\lfloor \frac{n}{t+1}\right\rfloor+1\right)=1$.
\end{proof} 

For example, $$p^{(2)}(15,3)=1+p^{(2)}(12,3)+p^{(2)}(11,4)+p^{(2)}(10,5)=1+4+1+1=7\ .$$

The following theorem presents a recurrence relation that, applied consecutively, allows us to write the numbers $p^{(t)}(n,m)$ in terms of $p(n,m)$.

\begin{theorem}\label{th6D}
	Let $n$, $m$ and $t$ be positive integers, so that $n>t>1$ and $m\le \left\lfloor \frac{n}{t+1}\right\rfloor$. Then
	$$
	p^{(t)}(n,m) = p^{(t-1)}(n,m)-p^{(t-1)}(n-t,m)\ .
	$$
\end{theorem}

\begin{proof} We are to prove the theorem by induction on $n$. For $n=t+1$ we have $m=1$. Considering that $p^{(t)}(t+1,1)=2$, $p^{(t-1)}(1,1)=1$ and 
	$$p^{(t-1)}(t+1,1)=p^{(t-1)}(t,1)+p^{(t-1)}(t+1,2)=2+1=3\ ,$$
	the base case of induction is finished. We suppose that the relation $$p^{(t)}(n',m) = p^{(t-1)}(n',m)-p^{(t-1)}(n'-t,m)$$ is true for any integer $n'$, $t<n'<n$. By Theorem \ref{th6C} we can write
	\begin{eqnarray}
	p^{(t)}(n,m) &=& 1+\sum_{k=m}^{\left\lfloor n/(t+1)\right\rfloor}\left(p^{(t-1)}(n-k,k)-p^{(t-1)}(n-t-k,k)\right) \nonumber\\
	&=& 1+\sum_{k=m}^{\left\lfloor n/(t+1)\right\rfloor}p^{(t-1)}(n-k,k)-\sum_{k=m}^{\left\lfloor n/(t+1)\right\rfloor}p^{(t-1)}(n-t-k,k)\ ,\nonumber
	\end{eqnarray}
	and, on considering that 
	$$p^{(t-1)}(n-k,k)=1\ , \mbox{\quad for\quad} \left\lfloor \frac{n}{t+1}\right\rfloor<k\le\left\lfloor \frac{n}{t}\right\rfloor\ ,$$ 
	and 
	$$p^{(t-1)}(n-t-k,k)=1\ , \mbox{\quad for\quad} \left\lfloor \frac{n}{t+1}\right\rfloor<k\le\left\lfloor \frac{n-t}{t}\right\rfloor\ ,$$
	we get
	$$p^{(t)}(n,m)=\sum_{k=m}^{\left\lfloor n/t\right\rfloor}p^{(t-1)}(n-k,k)-\sum_{k=m}^{\left\lfloor (n-t)/t\right\rfloor}p^{(t-1)}(n-t-k,k)\ .$$
	According to Theorem \ref{th6C}, the proof is finished.
	%$$p^{(t)}(n,m)=p^{(t-1)}(n,m)-p^{(t-1)}(n-t,m) .$$
	%Thus the theorem is proved.
\end{proof}

For instance, $$p^{(3)}(15,3)=p^{(2)}(15,3)-p^{(2)}(12,3)=7-4=3\ .$$ The number $15$ has three partitions with parts at least as large as $3$ in which the first part is thrice the second part: $[15]$, $[12,3]$ and $[9,3,3]$.

The $p^{(t)}(n, m)$ numbers are then also a direct generalization of the \textit{terminal compositions} developed by Kelleher \cite[section 5.4.1]{Kel06}.

We denote by $p^{(t)}(n)$ the number of partitions $\lambda \vdash n$ that have the property $\lambda_1\ge t\cdot\lambda_2$. 
%If we consider that partition $[n]$ has this property (the second part from the partition is $0$), then $p^{(t)}(n)=1+p^{(t)}(n,1)$. 
It is clear that $p^{(t)}(n)=p^{(t)}(n,1)$.
%The following corollary is a direct consequence of Theorem \ref{th6D}.
We then immediately have the following result.

\begin{corollary}\label{th6E}
	Let $n$ and $t$ be positive integers, so that $n>t>1$. Then 
	$$
	p^{(t)}(n) = p^{(t-1)}(n)-p^{(t-1)}(n-t)\ .
	$$
\end{corollary}

%The following two theorems are direct consequences of Theorem \ref{th6E}.
The following two corollaries then follows easily from Corollary \ref{th6E}.

\begin{corollary}\label{th6F}
	Let $n$ be a positive integer. Then, the number of partitions of $n$, that have the first part at least twice larger than the second part is $$p^{(2)}(n)=p(n)-p(n-2)\ .$$
\end{corollary}

%\begin{proof} Considering that $p^{(1)}(n)=1+p^{(1)}(n,1)=p(n,1)=p(n)$, for $n>2$ we apply Theorem \ref{th6E}. For $n=1$ or $n=2$, we take into consideration that $p(-1)=0$ and $p(0)=1$. Thus, we get $p^{(2)}(1)=p(1)-p(-1)=1$ and $p^{(2)}(2)=p(2)-p(0)=1$. 
%\end{proof}

For instance, 
$p^{(2)}(5)=p(5)-p(3)=7-3=4$, 
that means that number $5$ has four partitions that have the first part at least twice larger than the second part: $[5]$, $[4,1]$, $[3,1,1]$ and $[2,1,1,1]$. The sequence $p^{(2)}(n)$ appears in OEIS \cite{Slo11} as $A027336$. Corollary \ref{th6F} is known and can be found in Kelleher \cite[Corollary 4.1]{Kel09}.

\begin{corollary}\label{th6G}
	Let $n$ be a positive integer. Then, the number of partitions of $n$, that have the first part at least thrice larger than the second part is $$p^{(3)}(n)=p(n)-p(n-2)-p(n-3)+p(n-5)\ .$$
\end{corollary}

\begin{proof} For $n\in\left\{1,2,3\right\}$ the formula is immediately proved. For $n>3$ we apply Corollary  \ref{th6E} and Corollary \ref{th6F}.
\end{proof}

For example, 
$p^{(3)}(5)=p^{(2)}(5)-p^{(2)}(2)=4-1=3$, 
which means that number $5$ has three partitions that have the first part at least thrice larger than the second part:  $[5]$, $[4,1]$ and $[3,1,1]$. The sequence $p^{(3)}(n)$ appears in OEIS \cite{Slo11} as $A121659$.

The next theorem allows us to deduce an upper bound for $p^{(3)}(n)$.

\begin{theorem}\label{th6H}
	Let $n$ be a positive integer. Then  
	$$p(n)\ \le\ p(n-1)+p(n-2)-p(n-5)\ .$$
\end{theorem}

\begin{proof} (G. E. Andrews, Personal Communication) To prove the theorem we use Euler's formula \cite{And67}
	$$\sum_{n=0}^{\infty}\frac{z^n}{\left(q;q\right)_n}=\frac{1}{\left(z;q\right)_\infty}\ ,$$
	%$$\left(A;q\right)_n=\prod_{k=0}^{n-1}\left(1-Aq^{k}\right)$$
	where
	$$\left(A;q\right)_n=\left(1-A\right)\left(1-Aq\right)\cdots(1-Aq^{n-1})\ .$$
	
	We need to show that, except for the constant term, all the coefficients in
	$$\frac{1-q-q^2+q^5}{\left(q;q\right)_\infty}$$
	are non-positive. 
	
	We have
	\begin{eqnarray*}
		\frac{1-q-q^2+q^5}{\left(q;q\right)_\infty} &=& \frac{(1-q)(1-q^2)-q^3(1-q^2)}{\left(q;q\right)_\infty} \\
		&=& \frac{1}{\left(q^3;q\right)_\infty}-\frac{q^3}{(1-q)\left(q^3;q\right)_\infty}\\
		&=& \sum_{n=0}^{\infty}\frac{q^{3n}}{\left(q;q\right)_n}-\sum_{n=0}^{\infty}\frac{q^{3n+3}}{(1-q)\left(q;q\right)_n}\\
		&=&1+\sum_{n=1}^{\infty}\frac{q^{3n}}{\left(q;q\right)_n}-\sum_{n=1}^{\infty}\frac{q^{3n}\left(1-q^n\right)}{(1-q)\left(q;q\right)_n}\\
		&=&1+\sum_{n=1}^{\infty}\frac{q^{3n}}{\left(q;q\right)_n}\left(1-\frac{1-q^n}{1-q}\right)\\
		&=&1+\sum_{n=2}^{\infty}\frac{q^{3n}}{\left(q;q\right)_n}\left(-q-q^2-\cdots-q^{n-1}\right)
	\end{eqnarray*}
	and clearly the coefficient of $q^0$ is $1$, the coefficients of $q^k$ for $1\le k\le 6$ are all $0$, and for $k>6$ all the coefficients are negative.   
\end{proof}

We can then present the following corollary of Theorem \ref{th6H}.

\begin{corollary}\label{th6I}
	Let $n$ be a positive integer. Then
	$$p^{(3)}(n)\le p^{(2)}(n-1)\ .$$
\end{corollary}

\begin{proof} To prove this inequality we apply Corollary \ref{th6E} and Theorem \ref{th6I}.
\end{proof}

%%%%%%%%%%%%%%%%%%%%%%%%%%%%%%%%%%%%%%%%%%%%%%%%%%%%%%%%%%%%

\section{Traversal of partition tree}

To generate all the integer partitions of $n$ we can traverse in depth-first order the partition tree. To do this we need efficient traversal algorithms. Generally, for inorder traversing a binary tree we can efficiently use a stack. Strict binary trees are particular cases of binary trees and for their inorder traversal we can use Algorithm \ref{alg4}, built as algorithm 
\textit{INORD12} presented by Livovschi and Georgescu \cite[pp. 66]{Liv86}.

Analyzing Algorithm \ref{alg4}, we note that in the stack $S$ there are pushed only those nodes that are not leaf nodes in the strict binary tree for storing integer partition of $n$. 
As this strict binary tree has $2p(n)-1$ nodes, we deduce that algorithm \ref{alg4} executes $p(n)-1$ operations of pushing in the stack $S$ (line \ref{lin41}) and the same number of operations of popping from the stack $S$ (line \ref{lin42}). 
The leaf nodes are visited in the line \ref{lin43}, and the inner nodes are visited in the line \ref{lin44}. Thus, we deduce that the total number of iterations of the internal \textbf{while} loop from the lines \ref{alg45}-\ref{alg46} is $p(n)$.

\begin{algorithm}[H]
	\caption{Inorder traversal of strict binary tree}
	%\caption{InABS}
	\label{alg4}
	\begin{algorithmic}[1]
		\Require $root$
		\State initialize an empty stack $S$
		\State $v\leftarrow root$
		\State $c \leftarrow \textbf{true}$
		\While{$c$} \label{alg45}
		\While{node $v$ has left child}
		\State push node $v$ onto stack $S$ \label{lin41}
		\State $v\leftarrow$ left child of node $v$
		\EndWhile
		\State \textbf{visit} $v$ \label{lin43}
		\If{stack $S$ is not empty}
		\State pop node $v$ from stack $S$ \label{lin42}
		\State \textbf{visit} $v$ \label{lin44}
		\State $v\leftarrow$ right child of node $v$
		\Else
		\State $c \leftarrow \textbf{false}$
		\EndIf
		\EndWhile \label{alg46}
	\end{algorithmic}
\end{algorithm}

Strict binary trees for storing integer partitions are special cases of strict binary trees, while Algorithm \ref{alg4} is a general algorithm for inorder traversal of the strict binary trees. Adapting Algorithm \ref{alg4} to the special case of partition strict binary trees leads to more efficient algorithms for inorder traversal of these trees. Next we are to show how these algorithms can be obtained.

Let $(x,y)$ be a node from partition strict binary tree, so that $2x\le y$. 
If $(x_l,y_l)$ is the left child of the node $(x,y)$ and has the property $2x_l>y_l$ then, $(x_l,y_l)$ is the root of a strict binary subtree with exactly three nodes: $(x,y-x)$, $(y-x,0)$ and $(y,0)$. 
If $(x_r,y_r)$ is the right child of the node $(x,y)$ and has the property $2x_r>y_r$ then $(x_r,y_r)$ is the root of a strict binary subtree where all the left descendents are leaf nodes. 
This note allows us to inorder traverse the partition strict binary tree by the help of a stack in which we push only those nodes $(x,y)$ that have the property $2x\le y$. 
In this way, we get the Algorithm \ref{alg5} for inorder traversal of the partition strict binary tree of $n$.

The number of the inner nodes that have the form $(x,y)$ with the property $2x\le y$ is equal to the number of the partition of $n$ that have at least two parts where the first part is at least twice larger than the second part. 
%By Corollary \ref{th6F} it results that 
Algorithm \ref{alg5} executes $p^{(2)}(n)-1$ operations of pushing in the stack $S$ (line \ref{alg51}) and as many operations of popping from the stack $S$ (line \ref{alg52}). 
Any inner nodes that has the form $(x,y)$ with the property $2x>y$ has only a left child which is leaf node. 
A leaf node that has a left child is always visited in the line \ref{alg54} and, by the help of Corollary \ref{th6F}, we deduce that the number of these nodes is $p(n-2)$. 
A leaf node that is a right child is visited in the line \ref{alg59}. 
It results that the total number of iterations of the internal \textbf{while} loop from the lines \ref{alg57}-\ref{alg58} is $p^{(2)}(n)$.

\begin{algorithm}[H]
	\caption{Inorder traversal of partition strict binary tree  - version 1}
	%\caption{InAbsP1}
	\label{alg5}
	\begin{algorithmic}[1]
		\Require $n$
		\State initialize an empty stack $S$
		\State $(x,y)\leftarrow (1,n-1)$
		\State $c \leftarrow \textbf{true}$
		\While{$c$} \label{alg57}
		\While{$2x\le y$}
		\State push node $(x,y)$ onto stack $S$ \label{alg51}
		\State $(x,y)\leftarrow$ left child of node $(x,y)$ \label{alg53}
		\EndWhile
		\While{$x\le y$}
		\State \textbf{visit} $(y,0),(x,y)$ \label{alg54}
		\State $(x,y)\leftarrow$ right child of node $(x,y)$ \label{alg55}
		\EndWhile
		\State \textbf{visit} $(x+y,0)$ \label{alg59}
		\If{stack $S$ is not empty}
		\State pop node $(x,y)$ from stack $S$  \label{alg52}
		\State \textbf{visit} $(x,y)$ 
		\State $(x,y)\leftarrow$ right child of node $(x,y)$ \label{alg56}
		\Else
		\State $c \leftarrow \textbf{false}$
		\EndIf
		\EndWhile \label{alg58}
	\end{algorithmic}
\end{algorithm}

Thus, the reducing of the number of operations executed upon the stack allows to get a more efficient algorithm for traversing the partition strict binary tree. But the number of operations executed upon the stack could be reduced even more. 

Let $(x,y)$ be a node from the partition strict binary tree, so that $3x\le y$. 
If $(x_l,y_l)$ is the left child of the node $(x,y)$ and has the property $3x_l>y_l$, then $2x_l\le y_l$ and $(x_l,y_l)$ is the root of a strict binary tree in which the left subtree is a strict binary tree in which all the left children are leaf nodes. 
This note allows us to modify Algorithm \ref{alg5} in order to get a new algorithm for inorder traversing the partition strict binary trees. 
Thus, in Algorithm \ref{alg6} we push in the stack $S$ only those inner nodes $(x,y)$ that have the property $3x\le y$.

The number of the nodes $(x,y)$ with the property $3x\le y$ from the partition strict binary tree of positive integer $n$ is equal to the number of partitions of $n$ that have at least two parts where the first one is at least three times larger than the second one. 
In this way, Algorithm \ref{alg6} executes $p^{(3)}(n)-1$ operations of pushing in the stack $S$ (line \ref{alg61}) and as many operations of popping from the stack $S$ (line \ref{alg62}). 
In other words, the lines \ref{alg61}-\ref{alg63} are executed $p^{(3)}(n)-1$ times, in the same way as the lines \ref{alg62}-\ref{alg64}. 
The total number of iterations of the internal \textbf{while} loop from the lines \ref{alg65}-\ref{alg66} is $p^{(2)}(n)-p^{(3)}(n)$, namely $p^{(2)}(n-3)$. 
The leaf nodes that are right children are visited in the lines  \ref{alg67} and \ref{alg68}. The number of these nodes is $p^{(2)}(n)$. Considering that the line \ref{alg67} is executed $p^{(2)}(n-3)$ times, we deduce that the total number of iterations of the internal \textbf{while} loop from the lines \ref{alg69}-\ref{alg610} is $p^{(3)}(n)$.

\begin{algorithm}[H]
	\caption{Inorder traversal of partition strict binary tree  - version 2}
	\label{alg6}
	\begin{algorithmic}[1]
		\Require $n$
		\State initialize an empty stack $S$
		\State $(x,y)\leftarrow (1,n-1)$
		\State $c \leftarrow \textbf{true}$
		\While{$c$} \label{alg69}
		\While{$3x\le y$}
		\State push node $(x,y)$ onto stack $S$ \label{alg61}
		\State $(x,y)\leftarrow$ left child of node $(x,y)$ \label{alg63}
		\EndWhile
		\While{$2x\le y$} \label{alg65}
		\State \textbf{visit} $(y-x,0),(x,y-x)$ 
		\State $(p,q)\leftarrow$ right child of node $(x,y-x)$
		\While{$p\le q$}
		\State \textbf{visit} $(q,0),(p,q)$
		\State $(p,q)\leftarrow$ right child of node $(p,q)$
		\EndWhile
		\State \textbf{visit} $(p+q,0),(x,y)$ \label{alg67}
		\State $(x,y)\leftarrow$ right child of node $(x,y)$ 
		\EndWhile \label{alg66}
		\While{$x\le y$}
		\State \textbf{visit} $(y,0),(x,y)$ 
		\State $(x,y)\leftarrow$ right child of node $(x,y)$
		\EndWhile
		\State \textbf{visit} $(x+y,0)$ \label{alg68}
		\If{stack $S$ is not empty}
		\State pop node $(x,y)$ from stack $S$ \label{alg62}
		\State \textbf{visit} $(x,y)$ 
		\State $(x,y)\leftarrow$ right child of node $(x,y)$ \label{alg64}
		\Else
		\State $c \leftarrow \textbf{false}$
		\EndIf
		\EndWhile \label{alg610}
	\end{algorithmic}
\end{algorithm}

%%%%%%%%%%%%%%%%%%%%%%%%%%%%%%%%%%%%%%%%%%%%%%%%%%%%%%%%%%%%

\section{Algorithms for generating ascending compositions}

We showed above how we can use a stack for an efficient traversal of partition strict binary trees.  
In Algorithm \ref{alg4}, when visiting a node, the stack contains a part from the nodes that are on the path which connects that node to the root.  
Which are the nodes stored in the stack? 
A node is in a stack when another node is visited if and only if the visited node is the left child of the node from the stack. Taking into account the fact that partition strict binary trees have been obtained by converting partition trees, it is clear that, when visiting a leaf node, the content of the stack together with the leaf node represents a partition.

\begin{algorithm}[H]
	\caption{Generating ascending compositions - version 1}
	\label{alg7}
	\begin{algorithmic}[1]
		\Require $n$
		\State $k\leftarrow 0$
		\State $x\leftarrow 1$
		\State $y\leftarrow n-1$
		\State $c\leftarrow \textbf{true}$
		\While{$c$}
		\While{$2x\le y$}
		\State $k\leftarrow k+1$ \label{alg71}
		\State $a_k\leftarrow x$ 
		\State $y\leftarrow y-x$ \label{alg72}
		\EndWhile
		\While{$x\le y$}
		\State $k\leftarrow k+1$ 
		\State $a_k\leftarrow x$ 
		\State $k\leftarrow k+1$ 
		\State $a_k\leftarrow y$ 
		\State \textbf{visit} $a_1,a_2,\ldots,a_{k}$
		\State $k\leftarrow k-2$ 
		\State $x\leftarrow x+1$ \label{alg75}
		\State $y\leftarrow y-1$ \label{alg76}
		\EndWhile
		\State $k\leftarrow k+1$
		\State $a_{k}\leftarrow x+y$ 
		\State \textbf{visit} $a_1,a_2,\ldots,a_{k}$
		\State $k\leftarrow k-1$ 
		\If{$k>0$}
		\State $y\leftarrow x+y$ \label{alg73}
		\State $x\leftarrow a_k$ 
		\State $k\leftarrow k-1$ \label{alg74}
		\State $x\leftarrow x+1$ \label{alg75a}
		\State $y\leftarrow y-1$ \label{alg76a}
		\Else
		\State $c\leftarrow \textbf{false}$
		\EndIf
		\EndWhile 
	\end{algorithmic}
\end{algorithm}

To get an efficient algorithm for generating ascending compositions, we can replace the stack from Algorithm \ref{alg5} with an array that has nonnegative integer components $\left(a_1,a_2,\ldots,a_k\right)$, where $a_1$ represents the bottom of the stack and $a_k$ represents the top of the stack. 
If we give up visiting the inner nodes from Algorithm \ref{alg5} and the visit of a leaf node is preceded by the visit of the array $\left(a_1,a_2,\ldots,a_k\right)$ we get Algorithm \ref{alg7} for generating ascending compositions in lexicographic order.

Algorithm \ref{alg7} is presented in a form that allows fast identification both of the correlation between the operations executed in the stack $S$ and the operations executed with the array $\left(a_1,a_2,\ldots,a_k\right)$, and the movement operations in the tree. 
Thus, the lines \ref{alg71}-\ref{alg72} are responsible with the pushing of the nodes in the stack and the movement in the left subtree. 
The extraction of the nodes of the stack is realized in the lines \ref{alg73}-\ref{alg74}, and the movement in the right subtree is realized in the lines \ref{alg75}-\ref{alg76} and \ref{alg75a}-\ref{alg76a}. 
A slightly optimized version of Algorithm \ref{alg7} is presented in Algorithm \ref{alg9}.

\begin{algorithm}[H]
	\caption{Generating ascending compositions - version 2}
	\label{alg9}
	\begin{algorithmic}[1]
		\Require $n$
		\State $k\leftarrow 1$
		\State $x\leftarrow 1$
		\State $y\leftarrow n-1$
		\While{$k>0$} \label{alg91}
		\While{$2x\le y$} \label{alg93}
		\State $a_k\leftarrow x$ 
		\State $y\leftarrow y-x$ 
		\State $k\leftarrow k+1$ 
		\EndWhile \label{alg94}
		\State $t\leftarrow k+1$ 
		\While{$x\le y$} \label{alg95}
		\State $a_{k}\leftarrow x$ 
		\State $a_{t}\leftarrow y$ 
		\State \textbf{visit} $a_1,a_2,\ldots,a_{t}$
		\State $x\leftarrow x+1$ 
		\State $y\leftarrow y-1$ 
		\EndWhile \label{alg96}
		\State $y\leftarrow x+y-1$
		\State $a_{k}\leftarrow y+1$
		\State \textbf{visit} $a_1,a_2,\ldots,a_{k}$ \label{alg99}
		\State $k\leftarrow k-1$
		\State $x\leftarrow a_k+1$
		\EndWhile \label{alg92}
	\end{algorithmic}
\end{algorithm}

We note that Algorithm \ref{alg9} was rewritten  without using the \textbf{if} statement and the boolean variable $c$. 
It is obvious that the statements from the first branch of the \textbf{if} statement could not be eliminated, but they were rearranged around the \textbf{visit} statement (line \ref{alg99}). 
This was possible because the initializing statement of the variable $k$ was changed. 
As Algorithm \ref{alg9} executes less assignment statements and does not contain the \textbf{if} statement, the time required to execute Algorithm \ref{alg9} is less than the time required to execute Algorithm \ref{alg7}.

Comparing the algorithm \scshape AccelAsc \normalfont described and analyzed by Kelleher \cite{Kel06,Kel09} with Algorithm \ref{alg9}, we note that Algorithm \ref{alg9} is a slightly modified version of the algorithm \scshape AccelAsc\normalfont. 
In fact we have two presentations of the same algorithm. The differences are determined by the different methods of getting them.

The notes made upon the number of operations executed in Algorithm \ref{alg5} allow us to determine how many times the assignment statement are executed and how many times the boolean expressions are evaluated in Algorithm \ref{alg9}.

\begin{theorem}\label{th36}
	Algorithm \ref{alg9} executes $4p(n)+4p^{(2)}(n)$ assignment statements and evaluates $p(n)+3p^{(2)}(n)$ boolean expressions. 
\end{theorem}  

\begin{proof} 
	The total number of iterations of the internal \textbf{while} loop from the lines \ref{alg91}-\ref{alg92} is $p^{(2)}(n)$, and the \textbf{while} loop contains $5$ assignment statements. 
	The total number of iterations of the internal \textbf{while} loop from the lines \ref{alg93}-\ref{alg94} is $p^{(2)}(n)-1$, and the \textbf{while} loop contains $3$ assignment statements. 
	The total number of iterations of the internal \textbf{while} loop from the lines \ref{alg95}-\ref{alg96} is $p(n-2)$, and the \textbf{while} loop contains $4$ assignment statements. 
	Taking into account the first 3 assignment statements from the algorithm, we get the relation with which we determine the number of executions of assignment statements: $8p^{(2)}(n)+4p(n-2)$.
	The boolean expression from the line \ref{alg91} is evaluated $p^{(2)}(n)+1$ times, the boolean expression from the line \ref{alg93} is evaluated $p^{(2)}(n)+p^{(2)}(n)-1$ times, and the boolean expression from the line \ref{alg95} is evaluated $p^{(2)}(n)+p(n-2)$ times.
	Thus, it results that the boolean expressions from the algorithm are evaluated $4p^{(2)}(n)+p(n-2)$ times. 
	According to Corollary \ref{th6F}, the proof is finished.
\end{proof}

If in Algorithm \ref{alg6} we execute the conversions made in Algorithm \ref{alg5}, we get Algorithm \ref{alg10} for generating ascending composition in lexicographic order.

\begin{algorithm}[H]
	\caption{Generating ascending compositions - version 3}
	\label{alg10}
	\begin{algorithmic}[1]
		\Require $n$
		\State $k\leftarrow 1$
		\State $x\leftarrow 1$
		\State $y\leftarrow n-1$
		\While{$k>0$} \label{alg101}
		\While{$3x\le y$} \label{alg103}
		\State $a_k\leftarrow x$ 
		\State $y\leftarrow y-x$ 
		\State $k\leftarrow k+1$ 
		\EndWhile \label{alg104}
		\State $t\leftarrow k+1$ 
		\State $u\leftarrow k+2$ 		
		\While{$2x\le y$} \label{alg105}
		\State $a_k\leftarrow x$ 
		\State $a_t\leftarrow x$ 
		\State $a_u\leftarrow y-x$ 
\algstore{myalg}
\end{algorithmic}
\end{algorithm}

\begin{algorithm}[H]                   
\begin{algorithmic} 
\algrestore{myalg}			
		\State \textbf{visit} $a_1,a_2,\ldots,a_{u}$
		\State $p\leftarrow x+1$ 
		\State $q\leftarrow y-p$ 
		\While{$p\le q$} \label{alg107}
		\State $a_t\leftarrow p$ 
		\State $a_u\leftarrow q$ 
		\State \textbf{visit} $a_1,a_2,\ldots,a_{u}$
		\State $p\leftarrow p+1$ 
		\State $q\leftarrow q-1$ 
		\EndWhile \label{alg108}
		\State $a_t\leftarrow y$ 
		\State \textbf{visit} $a_1,a_2,\ldots,a_{t}$
		\State $x\leftarrow x+1$ 
		\State $y\leftarrow y-1$ 
		\EndWhile \label{alg106}
		\While{$x\le y$} \label{alg109}
		\State $a_{k}\leftarrow x$ 
		\State $a_{t}\leftarrow y$ 
		\State \textbf{visit} $a_1,a_2,\ldots,a_{t}$
		\State $x\leftarrow x+1$ 
		\State $y\leftarrow y-1$ 
		\EndWhile \label{alg1010}
		\State $y\leftarrow x+y-1$
		\State $a_{k}\leftarrow y+1$
		\State \textbf{visit} $a_1,a_2,\ldots,a_{k}$
		\State $k\leftarrow k-1$
		\State $x\leftarrow a_k+1$
		\EndWhile \label{alg102}
	\end{algorithmic}
\end{algorithm}

The notes made upon the number of operations executed in Algorithm \ref{alg6} allow us to determine how many times the assignment statements are executed and how many times the boolean expressions are evaluated in Algorithm \ref{alg10}.

\begin{theorem}\label{th37}
	Algorithm \ref{alg10} executes $4p(n)+5p^{(3)}(n)$ assignment statements and evaluates $p(n)+4p^{(3)}(n)$ boolean expressions.
\end{theorem}  

\begin{proof} 
	The total number of iterations of the internal \textbf{while} loop from the lines \ref{alg101}-\ref{alg102} is $p^{(3)}(n)$, and the \textbf{while} loop contains $6$ assignment statements. 
	The total number of iterations of the internal \textbf{while} loop from the lines \ref{alg103}-\ref{alg104} is $p^{(3)}(n)-1$, and the \textbf{while} loop contains $3$ assignment statements. 
	The total number of iterations of the internal \textbf{while} loop from the lines \ref{alg105}-\ref{alg106} is $p^{(2)}(n)-p^{(3)}(n)$, and the \textbf{while} loop contains $8$ assignment statements. 
	The total number of iterations of the internal \textbf{while} loops from the lines \ref{alg107}-\ref{alg108} and \ref{alg109}-\ref{alg1010} are $p(n-2)-\left(p^{(2)}(n)-p^{(3)}(n)\right)$, and each of them contains $4$ assignment statements. 
	Taking into account the first 3 assignment statements from the algorithm, we get the relation with which we determine the number of executions of assignment statements: $5p^{(3)}(n)+4p^{(2)}(n)+4p(n-2)$.
	The boolean expression from the line \ref{alg101} is evaluated $p^{(3)}(n)+1$ times, the boolean expression from the line \ref{alg103} is evaluated $p^{(3)}(n)+p^{(3)}(n)-1$ times, the boolean expression from the line \ref{alg105} is evaluated $p^{(2)}(n)$ times, and the boolean expressions from the lines \ref{alg107} and \ref{alg109} are evaluated $p^{(3)}(n)+p(n-2)$.
	It results that the boolean expressions from the algorithm are evaluated $4p^{(3)}(n)+p^{(2)}(n)+p(n-2)$ times.
	By Corollary \ref{th6F} and Corollary \ref{th6G} the proof is finished.
\end{proof}

\begin{theorem}\label{th38}
	Algorithm \ref{alg10} is more efficient than Algorithm \ref{alg9}.     
\end{theorem}  

\begin{proof} According to Theorem \ref{th36} and Theorem \ref{th37}, it is sufficient to show that 
	$$4p^{(3)}(n)\le 3p^{(2)}(n)\ .$$ 
	By Corollary \ref{th6E}, we can rewrite this inequality as
	$p^{(2)}(n)\le 4p^{(2)}(n-3)$.
	Using Corollary \ref{th6E} and Corollary \ref{th6I}, we obtain
	$$p^{(2)}(n)\le2p^{(2)}(n-3)+p^{(2)}(n-4)+p^{(2)}(n-5)\ .$$
	Taking into account $p^{(2)}(n-1)\le p^{(2)}(n)$, the theorem is proved.
\end{proof}

\section{Experimental results and conclusions}

We denote by $r_1(n)$ the ratio of the number of assignment statements executed by Algorithm \ref{alg10} and the number of assignment statements executed by Algorithm \ref{alg9},
$$r_1(n) = \frac{p(n)+1.25\cdot p^{(3)}(n)}{p(n)+p^{(2)}(n)}\ ,$$
and by $r_2(n)$ the ratio of the number of boolean expressions evaluated by Algorithm \ref{alg10} and the number of boolean expressions evaluated by Algorithm \ref{alg9},
$$r_2(n) = \frac{p(n)+4p^{(3)}(n)}{p(n)+3p^{(2)}(n)}\ .$$

We can see the evolution of the ratios $r_1(n)$ and $r_2(n)$, for $1<n\le 1500$. The graph is realized in Maple \cite{Gar01} and allows to note that the best performances of Algorithm \ref{alg10} compared to Algorithm \ref{alg9} appears for $50\le n\le 150$.

\begin{figure}[h]
	\caption{Evolution of performances of Algorithm \ref{alg10} compared to Algorithm \ref{alg9}}
	\label{fig10}
	\begin{center}
		\leavevmode
		\includegraphics[scale=0.9]{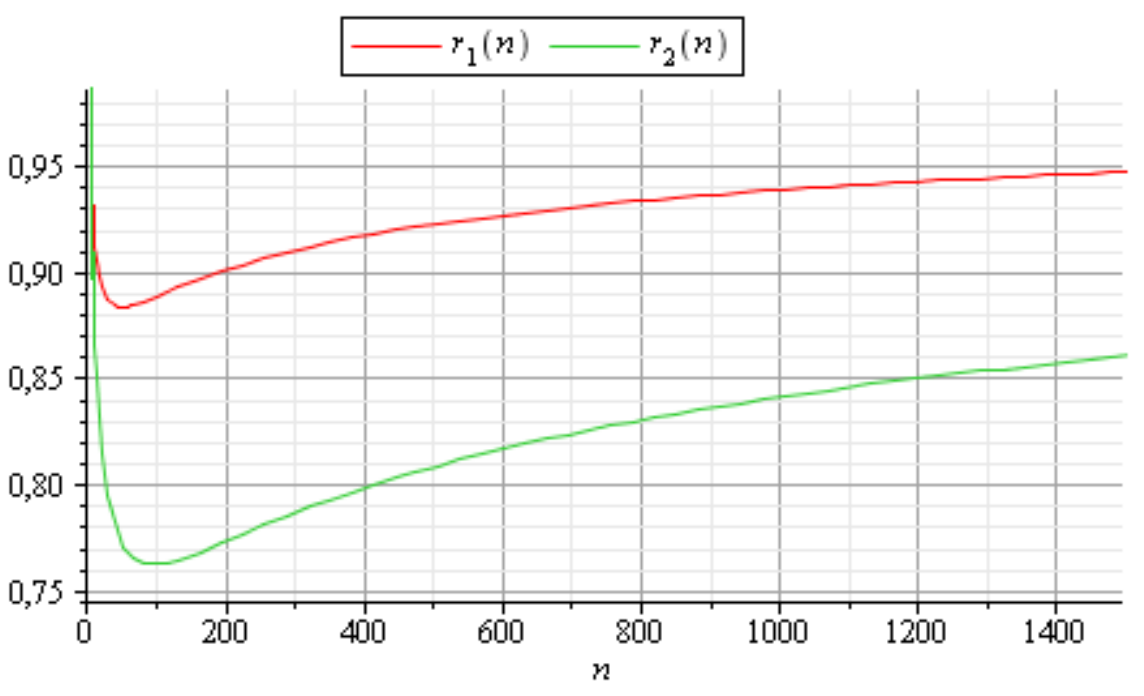}
	\end{center}
\end{figure}

Now let us measure CPU time for few values of $n$. 
To do this we encode the algorithms in C++ and the programs obtained with Visual C++ 2010 Express Edition will be run on three computers in similar conditions. 
The processor of the first computer is Intel Pentium Dual T3200 2.00 GHz, the processor of the second computer is Intel Celeron M 520 1.60GHz, and the the processor of third one is Intel Atom N270 1.60GHZ. 

CPU time is measured when the program runs without printing out ascending compositions. We denote by $t_1(n)$ the average time for Algorithm \ref{alg10} obtained after ten measurements, by $t_2(n)$ the average time for Algorithm \ref{alg9} obtained after ten measurements, and by $r(n)$ the ratio of $t_1(n)$ and $t_2(n)$, 
$$r(n)=\frac{t_1(n)}{t_2(n)}\ .$$ 
In Table \ref{talg910} we present the results obtained after the measurements made and also the ratios $r_1(n)$ and $r_2(n)$.

\begin{table}[H]
	\caption{Experimental results on few Intel CPU}
	\label{talg910}
	\centering
	\begin{tabular}{c|c|c|c|c|c|}
		\cline{2-6}
		& \multicolumn{2}{|c|}{Theoretical} & \multicolumn{3}{|c|}{Experimental $r(n)$} \\
		\cline{1-6}
		\multicolumn{1}{|c|}{$n$} & $r_1(n)$ & $r_2(n)$ & Pentium Dual& Intel Celeron & Intel Atom \\
		\hline
		\multicolumn{1}{|c|}{20} & 0.89556 & 0.82113 & 0.82257 & 0.78651 & 0.93977 \\
		\multicolumn{1}{|c|}{30} & 0.88738 & 0.79381 & 0.74365 & 0.75669 & 0.92693 \\
		\multicolumn{1}{|c|}{40} & 0.88467 & 0.77992 & 0.74621 & 0.75323 & 0.93712 \\
		\multicolumn{1}{|c|}{50} & 0.88403 & 0.77197 & 0.74905 & 0.76713 & 0.83457 \\
		\multicolumn{1}{|c|}{60} & 0.88438 & 0.76731 & 0.76146 & 0.76397 & 0.89760  \\
		\multicolumn{1}{|c|}{70} & 0.88525 & 0.76465 & 0.76036 & 0.76893 & 0.89812  \\
		\multicolumn{1}{|c|}{80} & 0.88639 & 0.76326 & 0.75739 & 0.77172 & 0.90994  \\
		\multicolumn{1}{|c|}{90} & 0.88766 & 0.76271 & 0.75549 & 0.77064 & 0.88617  \\
		\multicolumn{1}{|c|}{100} & 0.88901 & 0.76274 & 0.75464 & 0.76323 & 0.89064  \\
		\multicolumn{1}{|c|}{110} & 0.89037 & 0.76319 & 0.75352 & 0.77381 & 0.88936 \\
		\multicolumn{1}{|c|}{120} & 0.89174 & 0.76392 & 0.75358 & 0.77326 & 0.89617 \\
		\multicolumn{1}{|c|}{130} & 0.89308 & 0.76485 & 0.75445 & 0.77012 & 0.89004 \\
		\hline
	\end{tabular}
\end{table}

Analyzing the data presented in Table \ref{talg910}, we realize that the ratio $r_2(n)$ is a good approximation of the ratio $r(n)$, obtained experimentally on computers with Intel Pentium Dual or Intel Celeron processors.
We note that the ratio $r(n)$ obtained on the computer with Intel Atom processors is well approximated by the ratio $r_1(n)$.

Algorithm \ref{alg10} is the fastest algorithm for generating ascending compositions in lexicographic order in standard representation and can be considered an accelerated version of the algorithm \scshape AccelAsc \normalfont that Kelleher \cite{Kel06,Kel09} presented. 

\paragraph{Acknowledgements.} The author is greatly indebted to Professor George E. Andrews from the Department of Mathematics
	of The Pennsylvania State University who offered us the proof for Theorem 4 in a private context.
	The author would like to express his gratitude for the careful reading and helpful remarks to Oana Merca, which have resulted in what is hopefully a clearer paper. Moreover, the author wishes to thank the referee for his/her useful suggestions.


\begin{thebibliography}{00}


\bibitem{And67}
G. E. Andrews, Enumerative proofs of certain q-identities, Glasgow Math. J. \textbf{8}(1), 33-40, (1967) 

\bibitem{And76} 
G. E. Andrews, The Theory of Partitions, Addison-Wesley Publishing, (1976).

\bibitem{Gar01}
F. Garvan, The Maple Book, Chapman \& Hall/CRC, Boca Raton, Florida, (2001).

\bibitem{Kel06} 
J. Kelleher, Encoding Partitions as Ascending Compositions, PhD thesis, University College Cork, (2006).

\bibitem{Kel09} 
J. Kelleher, and B. O'Sullivan, Generating All Partitions: A Comparison Of Two Encodings, Published electronically at arXiv:0909.2331, (2009).

\bibitem{Lin05} 
R. B. Lin, Efficient Data Structures for Storing the Partitions of Integers, The 22nd Workshop on Combinatorics and Computation Theory. Taiwan, (2005).

\bibitem{Liv86} 
L. Livovschi, and H. Georgescu, Sinteza \c{s}i analiza algoritmilor. Editura \c{S}tiin\c{t}ific\u{a} \c{s}i Enciclopedic\u{a}, Bucure\c{s}ti, (1986)

\bibitem{Slo11} 
N. J. A. Sloane, The On-Line Encyclopedia of Integer Sequences. Published electronically at http://oeis.org, (2011).


\end{thebibliography}
\end{document}